\documentclass[12pt]{article}

\usepackage{epsfig,amsmath,amsfonts,amssymb,latexsym,color,amsthm,hhline,multirow,subfigure,graphicx}
\usepackage{amsmath}
\usepackage{amsthm}
\usepackage{enumerate}
\usepackage{amsfonts}
\usepackage{verbatim}
\usepackage{graphicx}
\usepackage{float}

\newtheorem{theorem}{Theorem}[section]
\newtheorem{prop}{Proposition}[section]
\newtheorem{lemma}{Lemma}[section]

\newcommand{\E}{{\mathbb E}}

\newcommand {\PP}{{\mathbb P}}
\newcommand{\sss}{\scriptscriptstyle}
\newcommand{\sF}{\mathcal{F}(\mu,\mathbf{p}_{\sss L})}
\newcommand{\sFp}{\mathcal{F}(\mathbf{p}_{\sss L})}
\newcommand{\pL}{\mathbf{p}_{\sss L}}

\begin{document}

\title{The tail does not determine the size of the giant}\parskip=5pt plus1pt minus1pt \parindent=0pt
\author{Maria Deijfen\thanks{Department of Mathematics, Stockholm University; {\tt mia@math.su.se}} \and Sebastian Rosengren \thanks{Department of Mathematics, Stockholm University; {\tt rosengren@math.su.se}} \and Pieter Trapman \thanks{Department of Mathematics, Stockholm University; {\tt ptrapman@math.su.se}} }
\date{June 2018}
\maketitle

\begin{abstract}
\noindent The size of the giant component in the configuration model, measured by the asymptotic fraction of vertices in the component, is given by a well-known expression involving the generating function of the degree distribution. In this note, we argue that the distribution over small degrees is more important for the size of the giant component than the precise distribution over very large degrees. In particular, the tail behavior of the degree distribution does not play the same crucial role for the size of the giant as it does for many other properties of the graph. Upper and lower bounds for the component size are derived for an arbitrary given distribution over small degrees $d\leq L$ and given expected degree, and numerical implementations show that these bounds are close already for small values of $L$. On the other hand, examples illustrate that, for a fixed degree tail, the component size can vary substantially depending on the distribution over small degrees. 

\vspace{0.3cm}


\vspace{0.2cm}

\end{abstract}

\section{Introduction and results}

The configuration model is one of the simplest and most well-known models for generating a random graph with a prescribed degree distribution. It takes a probability distribution with support on the non-negative integers as input and gives a graph with this degree distribution as output. The model is very well studied and there are precise answers to many questions concerning properties of the model such as the threshold for the occurrence of a giant component \cite{JL,MR-95}, the asymptotic fraction of vertices in the largest component \cite{JL,MR-98}, diameter and distances in the supercritical regime \cite{dist_fv,dist_iv,diam}, criteria for the graph to be simple \cite{Svante} etc; see \cite[Chapter 7]{RemcoI} and \cite[Chapters 4-5]{RemcoII} for detailed overviews. Empirical networks often exhibit power law distributions, that is, the number of vertices with degree $d$ decays as an inverse power of $d$ for large degrees. For this reason, there has been a lot of attention on properties of the configuration model with this type of degree distribution. Here we focus on the size of the largest component  in the supercritical regime -- specifically, the asymptotic fraction of vertices in the giant component -- as a functional of the degree distribution. Our main message is that the distribution over small degrees is more important for the size of the largest component than the tail behavior of the degree distribution. While this is not surprising, in view of the general focus on degree tails in the literature, we think it deserves to be pointed out and elaborated on.\medskip

\textbf{The model and its phase transition}

To define the model, fix the number $n$ of vertices in the graph and let $F=\{p_d\}_{d\geq 0}$ be a probability distribution with support on the non-negative integers. Assign a random number $D_i$ of half-edges independently to each vertex $i=1,\ldots, n$, with $D_i\sim F$. If the total number of half-edges is odd, one extra half-edge is added to a uniformly chosen vertex. Then pair half-edges uniformly at random to create edges, that is, first pick two half-edges uniformly at random and join them into an edge, then pick two half-edges from the set of remaining half-edges and create another edge, and so on until all half-edges have been paired. The construction allows for self-loops and multiple edges between the same pair of vertices. However, if the degree distribution has finite mean, such edges can be removed without changing the asymptotic degree distribution, and if the second moment is finite, there is a strictly positive probability that the graph is simple; see e.g.\ \cite{Tom_Anders,Svante}. 

Write $\mu=\E[D]$ and $\nu=\E[D(D-1)]/\mu$, and assume throghout that $p_2\neq 1$. It is well-known that the threshold for the occurrence of a giant component in the configuration model is given by $\nu=1$: if $\nu>1$, then there is with high probability a unique giant component occupying a positive fraction $\xi$ of the vertices as $n\to\infty$, while if $\nu<1$, then the largest component grows sublinearly in $n$; see \cite{MR-95,JL}. To see this, consider an exploration of the graph starting from a uniformly chosen vertex and then proceeding via nearest neighbors. For large $n$, such an exploration can be approximated by a branching process, where the offspring (=degree) of the first vertex has distribution $F$. For vertices in later generations, their degrees are distributed according to a size biased version of $F$. Indeed, by construction of the graph, the vertices constitute the end-points of uniformly chosen half-edges, and the probability of encountering a vertex with degree $d$ is therefore proportional to $d$. Since we arrive at a vertex from one neighbor, the remaining number of neighbors -- corresponding to the offspring of the vertex -- has a down-shifted size biased distribution $\tilde{F}=\{\tilde{p}_d\}_{d\geq 0}$, defined by
\begin{equation}\label{eq:sb_down}
\tilde{p}_d=\frac{(d+1)p_{d+1}}{\mu}.
\end{equation}
Infinite survival in the approximating branching process corresponds to a giant component in the graph, and the critical parameter $\nu$ is easily identified as the mean of the distribution \eqref{eq:sb_down}. Let $\xi$ denote the asymptotic fraction of vertices in the largest component, throughout refered to as the size of the largest component. The asymptotic size $\xi$ is given by the survival probability in the two-stage branching process (this can fail when $p_2\neq 1$, see Remark 2.7 in \cite{JL}). Write $g(s)$ for the probability generating function for the degree distribution $F$ and note that the probability generating function for $\tilde{F}$ is given by $g'(s)/\mu$. Let $\tilde{z}$ denote the probability that a branching process with offspring distribution $\tilde{F}$ goes extinct. Then $\tilde{z}$ is the smallest non-negative solution to the equation $s=g'(s)/\mu$, and
\begin{equation}\label{eq:xi}
\xi=1-g(\tilde{z}).
\end{equation}
A comprehensive description of the above exploration process can be found e.g.\ in \cite[Chapter 4]{RemcoII}. As for notation, when we want to emphasize the role of a given distribution $F$ for the above quantities, we write $\xi_{\sss F}$ and $\tilde{z}_{\sss F}$ etc. Furthermore, we always equip quantities related to down-shifted size biased distributions with a wiggle-hat.\medskip

\textbf{Basic examples}

We will be interested in how the size $\xi$ of the giant component depends on properties of the degree distribution $F$. Despite the large interest in the configuration model in the context of network modeling, there has been surprisingly little work on this issue. One recent example however is \cite{Lasse}, where component sizes are compared when degree distributions are ordered according to various concepts of stochastic domination. We also mention \cite{Tom_Pieter}, where a distribution is identified that maximizes the size of the largest component in a percolated configuration graph for a given mean degree: this is achieved by putting all mass at 0 and two consecutive integers. Here, we will throughout restrict to the class of distributions with $p_0=0$, that is, to graphs without isolated vertices. We hence require that all vertices have a chance of being included in a giant component (if such a component exists), and do not investigate cases where the component size can be tuned by removing some fraction of the vertices.

First note that, when the mean $\mu$ is fixed, the critical parameter $\nu$ increases as the variance of the distribution increases, making it easier to form a giant component. This might lead one to suspect that the size of the giant component is also increasing in $\nu$. This however is not true, in fact it is typically the other way around, as elaborated on in \cite{Lasse}. To understand this, note that fixing the mean and increasing the variance implies that there will be more vertices with small degree in the graph. Vertices with small degree are those that may not be included in the giant component, which then becomes smaller. Consider a very simple example with $D\in\{1,2,3\}$ where the probability $p_1$ of degree 1 is varied and the probabilities $p_2$ and $p_3$ are tuned so that the mean is kept fixed. As $p_1$ increases, also the probability $p_3$ increases, implying a larger variance. Figure \ref{fig:no_tail}(a) shows a plot of the component size and the critical parameter against $p_1$ when $\mu=2.1$, and we see that the giant component shrinks from occupying all vertices to a fraction 0.85 of them, while the critical parameter increases linearly. Figure \ref{fig:no_tail}(b) shows a similar plot (with only the component size) when $D\in\{1,2,10\}$ and again $\mu=2.1$, and we see that the component size decreases from 1 to less than 0.65. Note that these examples also illustrate that the mean in itself does not determine the component size, since the mean is constant in both pictures.

\begin{figure}
\centering \mbox{\subfigure[$D\in\{1,2,3\}$]{\includegraphics[height=4.7cm]{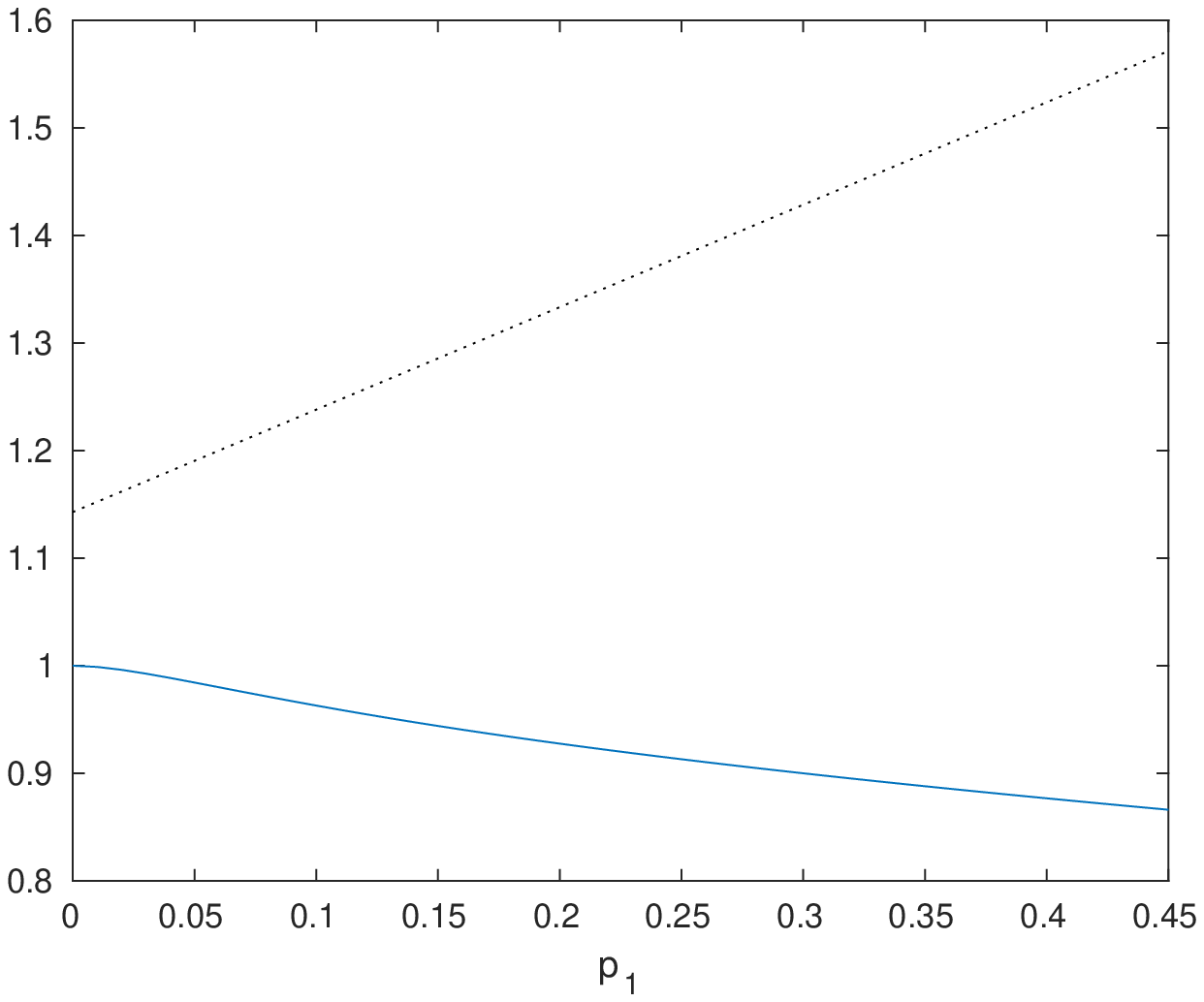}}}
\centering \mbox{\subfigure[$D\in\{1,2,10\}$]{\includegraphics[height=4.7cm]{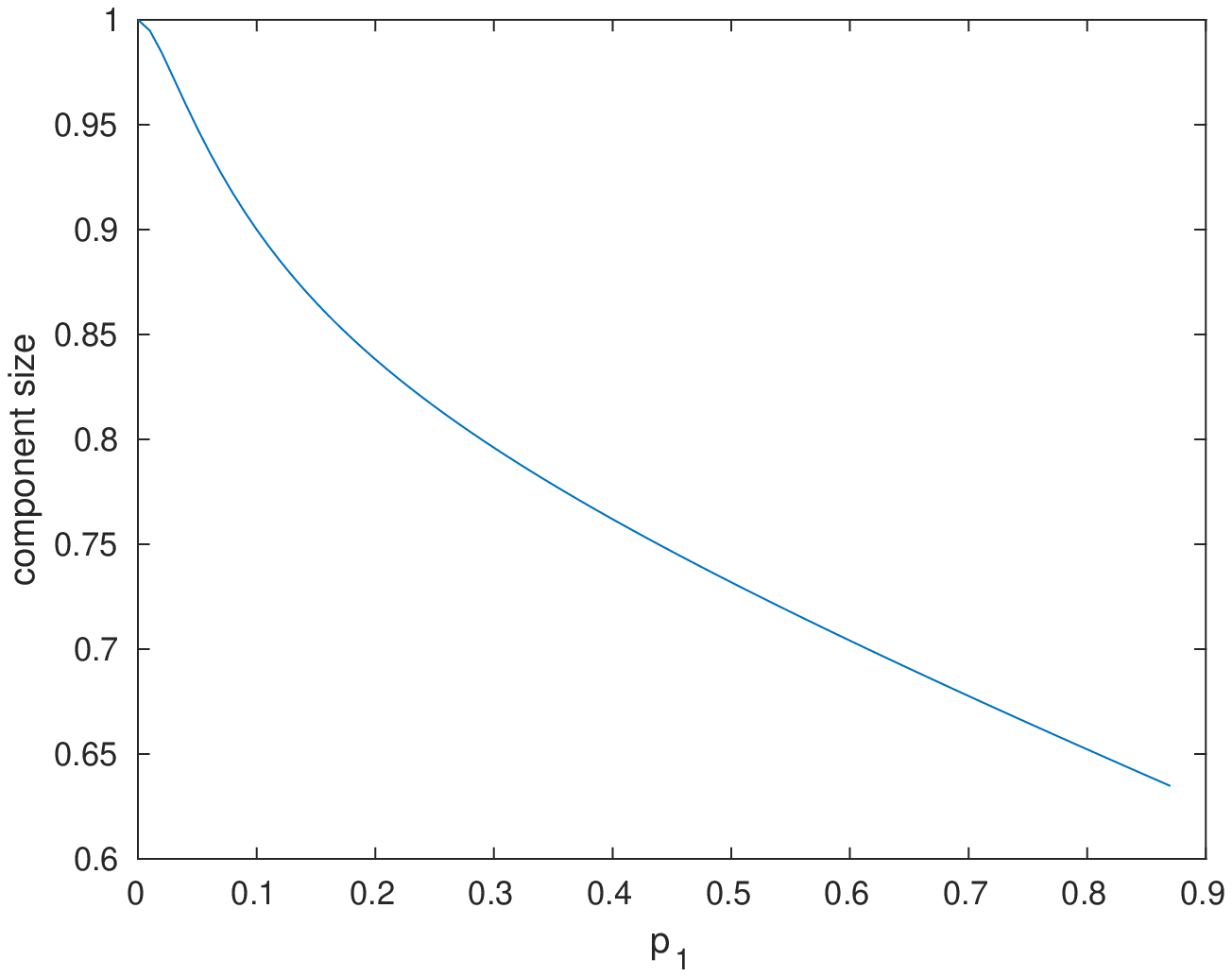}}}
\caption{Asymptotic size $\xi$ of the giant plotted aginst $p_1$ with mean fixed at $\mu=2.1$ for (a): $D\in\{1,2,3\}$ and (b): $D\in\{1,2,10\}$. In (a) also a plot of the critical parameter $\nu$ is included (while in (b) the critical parameter grows too large to fit in the plot). The probability $p_1$ does not run all the way to 1 since the mean cannot be preserved for large values of $p_1$.}\label{fig:no_tail}
\end{figure}


In the example we see that the component size $\xi$ decreases as the fraction of degree 1 vertices increases. This is natural since degree 1 vertices serve as dead ends in the component. If $\PP(D\geq 2)=1$ (and $p_2\neq 1$), then the extinction probability $\tilde{z}$ equals 0, implying that $\xi=1$. The size of the giant is hence determined by the balance between degree 1 vertices and vertices of larger degree. Increasing the variance in a distribution with a fixed mean typically implies an increase in the number of low degree vertices, and our main message is that the distribution over small degrees is in fact more important for the size of the giant component than the precise distribution over very large degrees. In particular, the tail behavior of the degree distribution does not play the same crucial role for the size of the giant as it does for certain other quantities such as e.g.\ the scaling of the distances in the giant component \cite{dist_fv,dist_iv}.

That the distribution over small degrees can play a significant role is illustrated in Figure \ref{fig:tail}, where the degrees have a fixed tail distribution and the remaining probability is allocated at small degrees in different mean-preserving ways. In Figure \ref{fig:tail}(a), the degree distribution is fixed for $d\geq 4$ (we consider a Poisson(2) distribution and a power-law with exponent -3) and the remaining probability is allocated at the degrees 1, 2 and 3. Specifically, the probability $p_1$ is varied and $p_2$ and $p_3$ are then adjusted so that the mean is kept fixed at $\mu=2.2$. Figure \ref{fig:tail}(a) shows plots of the component size against $p_1$ and we see that, although the tails remain the same, the component size changes with $p_1$ in both cases. Figure \ref{fig:tail}(b) shows a similar plot when the tail is fixed for $d\geq 11$ (Poisson and power-law) and the mean is equal to 3.5.

\begin{figure}
\centering \mbox{\subfigure[Blue: $p_d=\PP(\mbox{Po}(2)=d)$\newline \hspace*{0.6cm}Red: $p_d=2d^{-3}$]{\includegraphics[height=4.7cm]{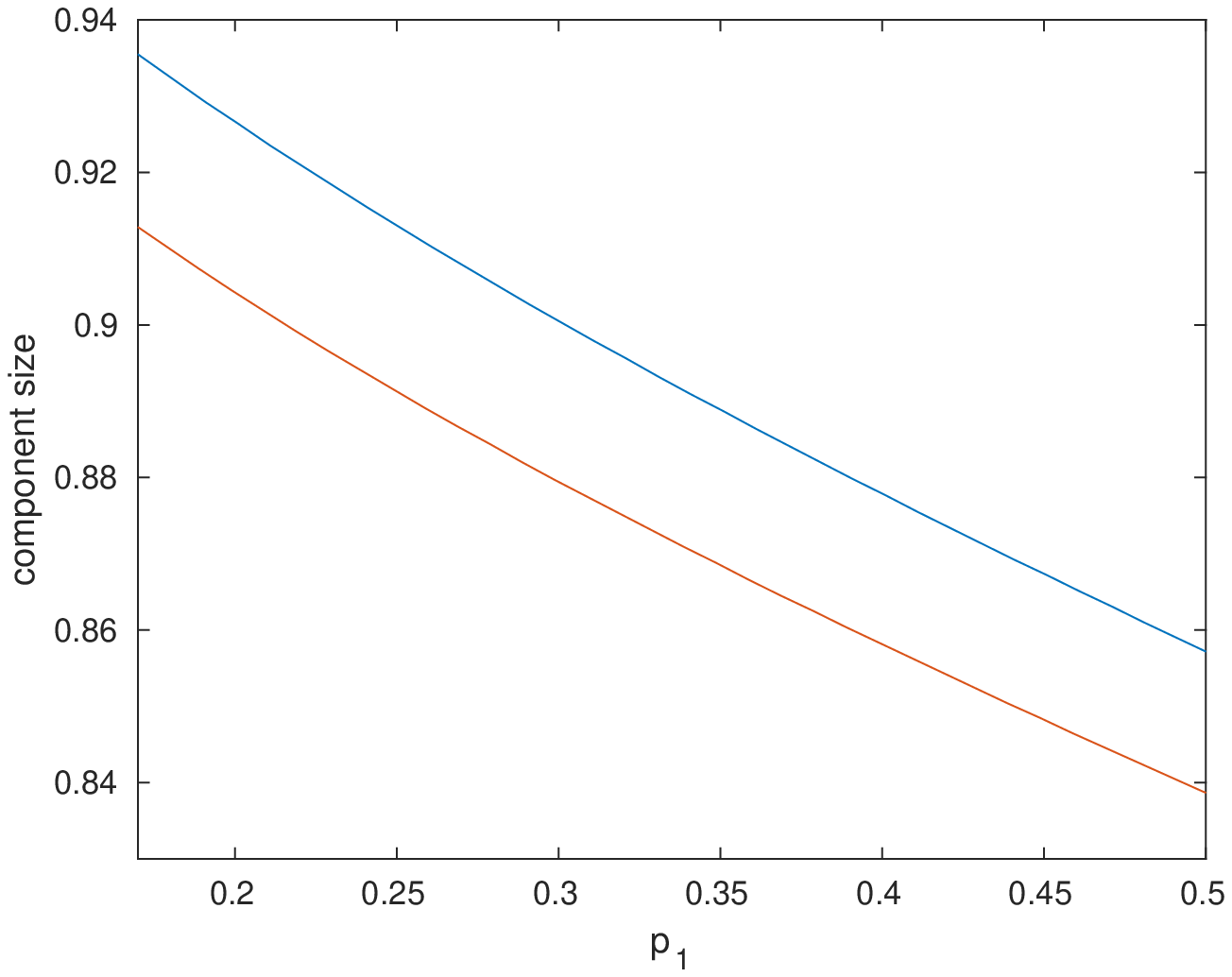}}}
\centering \mbox{\subfigure[Blue: $p_d=\PP(\mbox{Po}(7)=d)$\newline \hspace*{0.6cm}Red: $p_d=5d^{-2.5}$]{\includegraphics[height=4.7cm]{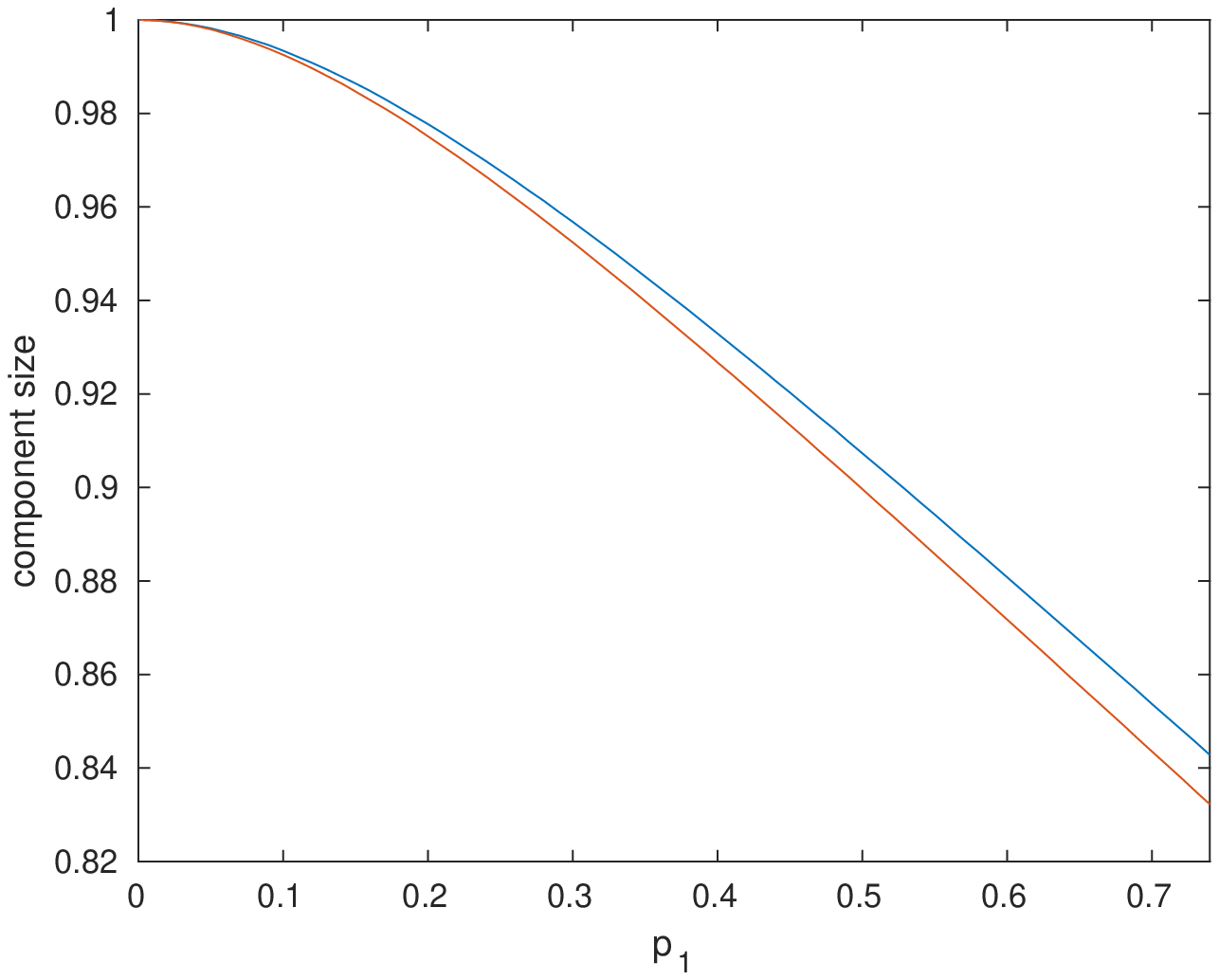}}}
\caption{Asymptotic size $\xi$ of the giant plotted against $p_1$ with (a): $p_d$ fixed for $d\geq 4$ and mean $\mu=2.2$ and (b): $p_d$ fixed for $d\geq 11$ and mean $\mu=3.5$. The constants in the power-law distributions are included to make the probabilities allocated in the tail roughly the same as for the Poisson distributions (approximately 0.1 in both cases).}\label{fig:tail}
\end{figure}

\textbf{Bounds for a given distribution over small degrees}\nopagebreak

We also argue that, conversely, fixing the distribution over small degrees typically leaves little room for controlling the component size by tuning the tail. Specifically, the difference between the maximal and the minimal achievable component size when the first $L$ probabilities and the mean are fixed tend to be small already for small values of $L$. This requires bounds for the component size for a given distribution over small degrees. To formulate our results here, let $\pL=\{p_1,\ldots,p_{\sss L}\}$ denote a fixed set of probabilities associated with degrees $1,\ldots,L$ for some $L\geq 1$, and write $\sFp$ for the set of all distributions having those specific initial probabilities. Also write $\sF$ for the set of all distributions in $\sFp$ with a given mean $\mu$. It turns out that a crude lower bound for the component size for distributions in $\sF$ is obtained by placing all remaining mass $p_{\sss >L}=1-\sum_{i=1}^Lp_i$ at the point $L+1$. Fixing also the mean $\mu$, under a mild technical condition, this bound can be modified into one that is optimal for distributions in $\sF$, that is, any larger bound is violated by some distribution in $\sF$. Under a similar technical condition, an optimal upper bound for distributions in $\sF$ is obtained by placing all remaining mass at two specific consecutive integers. 

For a fixed $\pL$, consider a distribution $G=G(\pL)$ with $p_{\sss L+1}=p_{\sss >L}$ (and $p_i=0$ for $i\geq L+2$), write $g_{\sss G}(s)$ for its probability generating function and $\xi_{\sss G}$ for the size of the giant component in a configuration graph with this degree distribution.

\begin{prop}\label{prop:lower}
For each fixed $\pL$, we have that $\xi_{\sss F}\geq \xi_{\sss G}$ for all $F\in\sFp$.
\end{prop}

Proposition \ref{prop:lower} is proved in the next section. To formulate (optimal) bounds for distributions in $\sF$, where also the mean $\mu$ is fixed, denote
$$
\kappa=\frac{1}{p_{>L}}\left(\mu-\sum_{d=1}^Ldp_d\right),
$$
and note that, for any $F\in\sF$, we have for $D\sim F$ that $\E[D|D>L]=\kappa$. Next, let $H=H(\mu,\pL)$ be a distribution where all remaining mass is placed at the two integers $\lfloor \kappa\rfloor$ and $\lceil\kappa\rceil$ (or one integer if $\kappa$ is an integer) in such a way that the mean is preserved, that is,
$$
p_d^{(\sss H)}=\left\{ \begin{array}{ll}
                        p_d & \mbox{for }d=0,\ldots,L;\\
                        (\lfloor \kappa\rfloor+1 -\kappa)p_{>L} & \mbox{for }d=\lfloor \kappa\rfloor;\\
                        (\kappa-\lfloor \kappa\rfloor)p_{>L} & \mbox{for }d=\lfloor \kappa\rfloor+1.
            \end{array}
            \right.
$$
Write $g_{\sss H}$ for the associated generating function and $\xi_{\sss H}$ for the component size in the corresponding configuration graph. Finally, let $\tilde{z}_{\sss G}$ and $\tilde{z}_{\sss H}$ denote the extinction probabilities in branching processes with offspring distributions given by down-shifted size biased versions of the above distributions. Our bounds on the component size with fixed initial probabilities $\pL$ and fixed mean $\mu$ are as follows.

\begin{theorem}\label{th:main} Fix $\pL$ and $\mu$.
\item[{\rm{(a)}}] If $\pL$ is such that $\tilde{z}_{\sss G}\leq e^{-\frac{1}{L+1}}$, then
    $$
    \xi_F\geq 1-g_{\sss G}\left(\tilde{z}_{\sss G}^{(\mu)}\right)\quad \mbox{for all }F\in\sF,
    $$
    where $\tilde{z}_{\sss G}^{(\mu)}$ is the smallest non-negative solution to the equation $s=g_{\sss G}'(s)/\mu$.
\item[{\rm{(b)}}] If $\pL$ and $\mu$ are such that $\tilde{z}_{\sss H}\leq e^{-\frac{2}{L+1}}$, then
    $$
    \xi_F\leq \xi_{\sss H}\quad\mbox{for all }F\in\sF.
    $$
The bounds are optimal under the given conditions, that is, in (a) we have that\newline 
$\inf_{F\in\sF}\xi_F=1-g_{\sss G}(\tilde{z}_{\sss G}^{(\mu)})$ and in (b) that $\sup_{F\in\sF}\xi_F=\xi_{\sss H}$. 
\end{theorem}

\textbf{Remark 1.} The restrictions on $\pL$ and $\mu$ are imposed for technical reasons. They imply that, if the extinction probabilities $\tilde{z}_{\sss G}$ and $\tilde{z}_{\sss H}$ are close to 1, then $L$ has to be large, that is, a sufficiently large part of the distribution has to be fixed. We believe that this serves to avoid e.g.\ situations where $\sF$ contains both subcritical and supercritical distributions. For most distributions, the conditions are mild, in the sense that they are satisfied already for moderate values of $L$ (in relation to $\mu$); see Table \ref{tab:bounds} for examples. Note however that, for $L=1$, when only the probability of degree 1 is fixed, the condition in (a) is not satisfied: in this case the distribution $G$ has mass only at 1 and 2 implying that $\tilde{z}_{\sss G}=1$.

\textbf{Remark 2.} The distribution $G$ can be thought of as the limiting case of a distribution $G_m$ where most of the remaining mass $p_{\sss >L}$ is placed at $L+1$ and a vanishing amount on another integer $m\to\infty$; see the proof of Theorem \ref{th:main}(b). The mean in this distribution $G_m$ is kept fixed at $\mu$, and the bound in (b) differs from the component size $\xi_{\sss G}$ obtained for the distribution $G$ in that the correct mean $\mu$ is used instead of the mean of $G$ in the equation defining $\tilde{z}_{\sss G}^{\sss(\mu)}$ (explaining the notation). Note that the spread in the distribution of the remaining mass is maximized in the distribution $G_m$. In the distribution $H$, on the other hand, the mass is concentrated as much as possible (while still keeping the mean fixed). 


\textbf{Numerical implementations}

Table 1 contains numerical values of the bounds in Proposition \ref{prop:lower} and Theorem \ref{th:main} for a few different distributions $\pL$ over small degrees (that all fulfill the technical conditions). As explained above, we only analyze distributions with $p_0=0$. We note that, in all cases, the upper and lower bound on the size of the giant are very close, supporting the claim that, if the distribution over low degrees is fixed, then the size of the giant is not affected much by the tail of the distribution. However, we would like to argue that this is the case for \emph{all} choises of $\pL$ and $\mu$ (satisfying the technical conditions) and for this we need to investigate the bounds more systematically. 

\begin{table}
\centering
\begin{tabular}{|l | l | l | l | l | l | l|}
\hline
$\pL=(p_1,\ldots,p_L)$ & $\mu$ & $L$ & $p_{\sss >L}$ & \footnotesize{Lower bound}  & \footnotesize{Lower bound} & \footnotesize{Upper bound} \\ 
& & & & \footnotesize{Proposition \ref{prop:lower}} & \footnotesize{Theorem \ref{th:main}(a)} & \footnotesize{Theorem \ref{th:main}(b)}\\ \hline
$(0.31,0.31,0.21)$ & 3 & 3 &  0.17 & 0.9140 & 0.9504 & 0.9508 \\
$(0.43,0.32)$ & 3 & 2 & 0.25 & 0.5896 & 0.9019 & 0.9103 \\
$(0.7,0,0)$ & 2 & 3 & 0.3 & 0.7023 & 0.7247 & 0.7318 \\
$(0.7,0,0)$ & 3 & 3 & 0.3 & 0.7023 & 0.8319 & 0.8366 \\
$(0.5,0.25,0.125)$ & 2 & 3 & 0.125 & 0.7047 & 0.7553 & 0.7680 \\
$(0.5,0.25,0.125)$ & 3 & 3 & 0.125 & 0.7047 & 0.8836 & 0.8851 \\
\hline
\end{tabular}
\caption{Bounds for the size of the giant component from Proposition \ref{prop:lower} and Theorem \ref{th:main}. The first two examples are Poisson probabilities with mean 2 and 1.5, respectively, conditional on the degree being strictly positive. All distributions satisfy the technical conditions in Theorem \ref{th:main}(a) and (b).}
\label{tab:bounds}
\end{table}

\begin{table}
\centering
\begin{tabular}{|l | l | l | l | l | l | l|}
\hline
$L$ & Maxdiff & $\pL=(p_1,\ldots,p_L)$ & $\mu$ & \footnotesize{Lower bound} & \footnotesize{Upper bound} \\ 
& & & & \footnotesize{Theorem \ref{th:main}(a)} & \footnotesize{Theorem \ref{th:main}(b)}\\ \hline
2 & 0.055 & (0.6, 0.1) & 2.0 & 0.7059 & 0.7616 \\
3 & 0.041 & (0.55, 0.35, 0) & 1.8 & 0.5664 & 0.6078 \\
4 & 0.029 & (0.75, 0.1, 0.1 0) & 1.6 & 0.4203 & 0.4494 \\
5 & 0.024 & (0.75, 0.2, 0, 0, 0) & 1.6 & 0.4188 & 0.4423 \\
\hline
\end{tabular}
\caption{Maximal difference between the bounds in Theorem \ref{th:main}(a) and (b) for different values of $L$. We also give the probabilities $\pL$ and mean $\mu$ that give rise to the maximal difference and the corresponding values of the bounds.}
\label{tab:maxdiff}
\end{table}

\begin{figure}
\centering \mbox{\includegraphics[height=5.5cm]{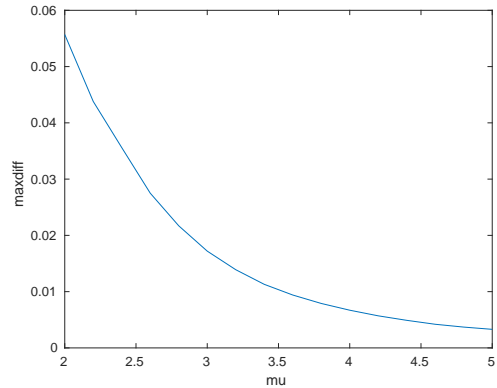}}
\caption{Maximal difference between the bounds in Theorem \ref{th:main}(a) and (b) plotted against $\mu$. The maximum is taken over $L\in\{2,3,4,5\}$ and distributions $\pL$.}\label{fig:mu_mxdiff}
\end{figure}

If a large part of the distribution is fixed, it is not surprising that the component size cannot be tuned much, and we hence focus on small values of $L$, say $L\leq 5$. For each $L\in\{2,3,4,5\}$ we have made a grid search (with step length 0.05) of all possible distributions $\pL$ for different values of $\mu\in[1,5]$ (with step length 0.2). Table 5 shows the maximal difference between the upper and lower bound for distributions fulfilling the technical conditions and also for which distribution $\pL$ and mean $\mu$ that this maximal difference is observed. We note that, for $L=2$, the maximal difference is 0.055 and it then decreases with $L$ to 0.024 for $L=5$. Througout, the worst cases occur for small values of $\mu$. This is confirmed by Figure \ref{fig:mu_mxdiff}, where the maximal difference (over $L\in\{2,3,4,5\}$ and $\pL$) is plotted against $\mu$. We remark that, in all cases, the maximal difference was observed for $L=2$. In summary, this indicates that, if the first $L=5$ probabilities are fixed (and the technical conditions satisfied), then the component size cannot vary more than approximately 0.024. 

It would of course be desirable to estimate the difference between the bounds analytically, but it seems complicated to obtain good estimates for small values of $L$, which is what we are after.

In the next section we prove Proposition \ref{prop:lower} and Theorem \ref{th:main}.

\section{Proof of Theorem \ref{th:main}}

Assume throughout this section that $\pL$ is fixed.

\begin{proof}[Proof of Proposition \ref{prop:lower}]
Fix a distribution $F\in\sFp$. Since the component size $\xi$ is given by \eqref{eq:xi}, and $\xi_{\sss G}$ by the analogous expression for the distribution $G$, we need to show that $g(\tilde{z})\leq g_{\sss G}(\tilde{z}_{\sss G})$. It is clear that $g(s)\leq g_{\sss G}(s)$ for any $s\in[0,1]$, and hence, since generating functions are increasing, it follows that $g(\tilde{z})\leq g_{\sss G}(\tilde{z}_{\sss G})$ if we show that $\tilde{z}\leq \tilde{z}_{\sss G}$. Let $\{\tilde{p}_d^{\sss (G)}\}_{d=0}^L$ denote the probabilities defining the down-shifted size biased version $\tilde{G}$ of $G$ and recall that $\{\tilde{p}_d\}$, defined in \eqref{eq:sb_down}, denote the corresponding probabilities for $F$. It is not hard to see that $\tilde{p}_d\leq \tilde{p}_d^{\sss (G)}$ for all $i=0,\ldots,L$ (and $\tilde{p}_i^{\sss (G)}=0$ for $i\geq L+1$). Hence $\tilde{G}$ is stochastically smaller than $\tilde{F}$, implying that $\tilde{z}\leq \tilde{z}_{\sss G}$, as desired.
\end{proof}

For the remainder of the section, we fix also the mean $\mu$.

\begin{proof}[Proof of Theorem \ref{th:main}(a)]
We begin by defining a sequence of distributions $\{G_m\}_{m\geq \kappa}$ where a vanishing (as $m\to\infty$) fraction of the remaining mass is placed at $m$ and the rest at $L+1$, in such a way that the mean of the distribution is fixed at $\mu$. Let 
$$
r_m=\frac{\kappa-(L+1)}{m-(L+1)}.
$$
Then $G_m=\{p_d^{\sss (m)}\}_{d\geq 1}$ is defined by
$$
p_d^{\sss (m)}=\left\{ \begin{array}{ll}
                        p_d & \mbox{for }d=0,\ldots,L;\\
                        (1-r_m)p_{\sss >L} & \mbox{for }d=L+1;\\
                        r_mp_{\sss >L} & \mbox{for }d=m.
            \end{array}
            \right.
$$
Note that $G_m\in\sF$. Write $\tilde{z}_m$ for the extinction probability of a branching process with offspring distribution given by a down-shifted size biased version $\tilde{G}_m$ of $G_m$. Also, let $\tilde{z}_{\sss G}^{\sss (\mu)}$ denote the smallest solution  of the equation $s=g'_{\sss G}(s)/\mu$. We will show that (i) $1-\xi_F = g_F(\tilde{z}_F)\leq g_{\sss G}(\tilde{z}_{\sss G}^{\sss (\mu)})$ for all $F\in\sF$ and then, in order to show that the bound is sharp, that (ii) $\tilde{z}_m$ is increasing for large $m$ and converges to $\tilde{z}_{\sss G}^{\sss (\mu)}$.

To establish (i), first fix a distribution $F\in\sF$, that is, in addition to $\pL$ we also fix $p_d$ for $d\geq L+1$ such that the mean is $\mu$. Since $g_{\sss F}(s)\leq g_{\sss G}(s)$ for all $s$ and generating functions are increasing, the desired conclusion follows if $\tilde{z}_{\sss F}\leq \tilde{z}_{\sss G}^{\sss (\mu)}$, which in turn follows if $g_{\sss F}'(\tilde{z}_{\sss F})\leq g'_{\sss G}(\tilde{z}_{\sss F})$, since the smallest solution $\tilde{z}_{\sss G}^{\sss (\mu)}$ of $s=g'_{\sss G}(s)/\mu$ must then be larger than $\tilde{z}_{\sss F}=g'_{\sss F}(\tilde{z}_{\sss F})/\mu$. The assumption $\tilde{z}_{\sss G}\leq e^{-\frac{1}{L+1}}$ ensures that functions of the form $f(d)=ds^{d-1}$, with $s\leq \tilde{z}_{\sss G}$, are strictly decreasing for $d\geq L+1$. Since $\tilde{z}_{\sss F}\leq \tilde{z}_{\sss G}$ (as shown in Proposition \ref{prop:lower}), this means that
$$
\sum_{d=L+1}^\infty d\tilde{z}_{\sss F}^{d-1}p_d\leq (L+1)\tilde{z}_{\sss F}
^L\sum_{d=L+1}^\infty p_d= (L+1)\tilde{z}_{\sss F}^{\sss L}p_{\sss >L},
$$
which implies that $g_{\sss F}'(\tilde{z}_{\sss F})\leq g'_{\sss G}(\tilde{z}_{\sss F})$, as desired.

As for (ii), note that it follows from the proof of Proposition \ref{prop:lower} that $\tilde{z}_m\leq \tilde{z}_{\sss G}$, and the assumption $\tilde{z}_{\sss G}\leq e^{-\frac{2}{L+1}}$ ensures that $\tilde{z}_{\sss G}<1$ so that $\tilde{z}_m<1$. The extinction probability $\tilde{z}_m$ solves the equation $s=g'_m(s)/\mu$ and hence it follows that $\tilde{z}_m$ is increasing for large $m$ if $g'_m(\tilde{z}_m)\leq g'_{m+1}(\tilde{z}_m)$ when $m$ is large -- indeed, the smallest solution $\tilde{z}_{m+1}$ of $s=g'_{m+1}(s)/\mu$ must then be larger than $\tilde{z}_m$. Noting that $g'_m(s)=\sum dp_d^{(m)}s^{d-1}$, we obtain that
\begin{equation*}
\begin{array}{cl}
\ & g'_{m+1}(\tilde{z}_m)-g'_m(\tilde{z}_m)\\
= & p_{\sss >L}
[(L+1)(r_m-r_{m+1}) (\tilde{z}_m)^{L} + (m+1)r_{m+1}(\tilde{z}_m)^{m}-m r_m (\tilde{z}_m)^{m-1}]\\
>  & p_{\sss >L} (\tilde{z}_m)^{L}
[(L+1)(r_m-r_{m+1})  - m r_m (\tilde{z}_m)^{m-L-1}],
\end{array}
\end{equation*}
which is positive for large $m$ since $r_m-r_{m+1}$ is positive and of order $m^{-2}$ while $m r_m (\tilde{z}_m)^{m-L-1}$ is exponentially decreasing in $m$ (recall, $\tilde{z}_m\leq  \tilde{z}_{\sss G}<1$ for all $m \geq \kappa$). Since $\tilde{z}_m$ is increasing for large $m$ and bounded from above by $\tilde{z}_{\sss G}<1$, it converges to some limit $\tilde{z}_\infty$ that is strictly smaller than 1. Furthermore, since $r_m \to 0$ and $\tilde{z}_m\leq  \tilde{z}_{\sss G}<1$, we obtain that
\begin{equation*}
\begin{array}{rcl}
\tilde{z}_{m} = g'_m(\tilde{z}_{m})/\mu & = &  \sum_{k=1}^L k p_k (\tilde{z}_{m})^{k-1} + p_{\sss >L}(L+1)(1-r_m) (\tilde{z}_{m})^{L} +p_{\sss >L}m r_m(\tilde{z}_{m})^{m-1}\\
\ & \to &  \sum_{k=1}^L k p_k (\tilde{z}_\infty)^{k-1} + p_{\sss >L}(L+1)\tilde{z}_\infty^L + 0 = g'_{\sss G}(\tilde{z}_\infty)/\mu
\end{array}
\end{equation*}
as $m \to \infty$. Therefore, $\tilde{z}_\infty$ is the unique solution of the equation $s=g'_{\sss G}(s)/\mu$ in $(0,1)$, which is also the definition of $\tilde{z}_{\sss G}^{\sss (\mu)}$.

Finally, we obtain that the derived bound, $1-\xi_F = g_F(\tilde{z}_F)\leq g_{\sss G}(\tilde{z}_{\sss G}^{\sss (\mu)})$, is optimal: Since $g_m(\tilde{z}_m)\nearrow g_{\sss G}(\tilde{z}_{\sss G}^{\sss (\mu)})$, for any $\xi>1-g_{\sss G}(\tilde{z}_{\sss G}^{\sss (\mu)})$ there exist an $m$ such that $\xi_{\sss G_m}<\xi$.
\end{proof} 

The following simple lemma will be used in the proof of Theorem \ref{th:main}(b). It will be applied to $N\stackrel{d}{=}D|D>L$ -- that is, a random variable distributed as $D$ conditional on being strictly larger than $L$ -- and the mean is therefore denoted by $\kappa$.

\begin{lemma}\label{le}
Let $N$ be an integer valued random variable with mean $\kappa$. There exist integer valued random variables $N_1$ and $N_2$ with $\E[N_1]=\lfloor\kappa\rfloor$ and $\E[N_2]=\lfloor\kappa\rfloor+1$ such that, with $Z\sim\mbox{Be}(\lfloor\kappa\rfloor+1-\kappa)$ independent of $N_1$ and $N_2$, we have that
$$
N\stackrel{d}{=}ZN_1+(1-Z)N_2.
$$
\end{lemma}

\begin{proof}
Let $N_{\rm{low}}\stackrel{d}{=}N|N\leq \lfloor\kappa\rfloor$ and $N_{\rm{hi}}\stackrel{d}{=}N|N>\lfloor\kappa\rfloor$ be independent and write $\kappa_{\rm{low}}$ and $\kappa_{\rm{hi}}$ for the respective means. Furthermore, let $X$ and $Y$ be Bernoulli variables independent of $N_{\rm{low}}$ and $N_{\rm{hi}}$ with parameter $\frac{\kappa_{\rm{hi}}-\lfloor\kappa\rfloor}{\kappa_{\rm{hi}}-\kappa_{\rm{low}}}$ and $\frac{\kappa_{\rm{hi}}-\lfloor\kappa\rfloor-1}{\kappa_{\rm{hi}}-\kappa_{\rm{low}}}$, respectively. Then set
$$
N_1=XN_{\rm{low}}+(1-X)N_{\rm{hi}}\qquad N_2=YN_{\rm{low}}+(1-Y)N_{\rm{hi}}.
$$
It is straightforward to confirm that $\PP(ZN_1+(1-Z)N_2=i)=\PP(N=i)$ for all $i$.
\end{proof}

\begin{proof}[Proof of Theorem \ref{th:main}(b).]
We need to show that $g_{\sss F}(\tilde{z}_{\sss F})\geq g_{\sss H}(\tilde{z}_{\sss H})$ for all $F\in\sF$. To this end, we begin by showing that
\begin{equation}\label{eq:kappa_dom}
g_{\sss F}(s)\geq g_{\sss H}(s)\quad\mbox{for all }s\in[0,1]\mbox{ and  all }F\in\sF.
\end{equation}
Pick $F\in\sF$ and let $D\sim F$. The probability generating function $G_{\sss F}(s)$ can be written as
$$
g_{\sss F}(s)=\E\left[s^D\right]=\E\left[s^D|D\leq L\right](1-p_{\sss >L})+\E\left[s^D|D>L\right]p_{\sss >L}.
$$
Applying Lemma \ref{le} with $N\stackrel{d}{=}D|D>L$, we can write
$$
\E\left[s^D|D>L\right]=\E\left[s^{N_1}\right]\PP(Z=1)+\E\left[s^{N_2}\right]\PP(Z=0),
$$
where $\E[N_1]=\lfloor\kappa\rfloor$, $\E[N_2]=\lfloor\kappa\rfloor+1$ and $Z\sim\mbox{Be}(\lfloor\kappa\rfloor+1-\kappa)$. Jensen's inequality then yields that
$$
\E\left[s^D|D>L\right]\geq s^{\lfloor\kappa\rfloor}\PP(Z=1)+s^{\lfloor\kappa\rfloor+1} \PP(Z=0).
$$
The probability generating function $g_{\sss H}(s)$ can be written as
$$
g_{\sss H}(s)=\E\left[s^D|D\leq L\right](1-p_{\sss >L})+\left(s^{\lfloor\kappa\rfloor}\PP(Z=1)+s^{\lfloor\kappa\rfloor+1}\PP(Z=0)\right)p_{\sss >L},
$$
and hence \eqref{eq:kappa_dom} follows. Since generating functions are increasing, the desired bound now follows if $\tilde{z}_{\sss F}\geq \tilde{z}_{\sss H}$, which in turn follows if $g'_{\sss F}(\tilde{z}_{\sss H})\geq g'_{\sss H}(\tilde{z}_{\sss H})$. The assumption $\tilde{z}_{\sss H}\leq e^{-\frac{2}{L+1}}$ ensures that this is the case: Let $D_\kappa\sim H$. Note that, since the two distributions $F$ and $H$ agree up to $L$, the desired inequality follows if $\E[D\tilde{z}_{\sss H}^D|D>L]\geq \E[D_\kappa\tilde{z}_{\sss H}^{D_\kappa}|D_\kappa>L]$. We have that
$$
\E\left[D_\kappa\tilde{z}_{\sss H}^{D_\kappa}|D_\kappa>L\right]=\lfloor\kappa\rfloor
\tilde{z}_{\sss H}^{\lfloor\kappa\rfloor}(\lfloor\kappa\rfloor+1-\kappa) + (\lfloor\kappa\rfloor+1)\tilde{z}_{\sss H}^{\lfloor\kappa\rfloor+1}(\kappa-\lfloor\kappa\rfloor).
$$
By Lemma \ref{le}, with $N\stackrel{d}{=}D|D>L$, we can write
$$
\E\left[D\tilde{z}_{\sss H}^D|D>L\right] = \E\left[N_1\tilde{z}_{\sss H}^{N_1}\right](\lfloor\kappa\rfloor+1-\kappa)+ \E\left[N_2\tilde{z}_{\sss H}^{N_2}\right](\kappa-\lfloor\kappa\rfloor),
$$
where $N_1,N_2>L$, $\E[N_1]=\lfloor\kappa\rfloor$ and $\E[N_2]=\lfloor\kappa\rfloor+1$. The assumption $\tilde{z}_{\sss H}\leq e^{-\frac{2}{L+1}}$ implies that $f(d)=d\tilde{z}_{\sss H}^{d}$ is convex for $d\geq L+1$. It follows from a straightforward modification of Jensen's inequality (specifically, a restriction to $[L+1,\infty)$) that
$$
\E\left[N_1\tilde{z}_{\sss H}^{N_1}\right]\geq \lfloor\kappa\rfloor\tilde{z}_{\sss H}^{\lfloor\kappa\rfloor},\qquad
\E\left[N_2\tilde{z}_{\sss H}^{N_2}\right]\geq (\lfloor\kappa\rfloor+1)\tilde{z}_{\sss H}^{\lfloor\kappa\rfloor+1}
$$
and the bound follows. That the bound is optimal follows by noting that $F_\kappa\in\sF$, that is, the distribution defining the bound is included in the class.
\end{proof}


\begin{thebibliography}{99}


\bibitem{Tom_Anders} Britton, T., Deijfen, M. and Martin-L\"{o}f, A. (2006): Generating simple random graphs with prescribed degree distribution, \emph{J. Stat. Phys.} \textbf{124}, 1377-1397.

\bibitem{Tom_Pieter} Britton, T. and Trapman, P. (2012): Maximizing the size of the giant, \emph{J. Appl.Probab.} \textbf{49}, 1156-1165.


\bibitem{RemcoI} van der Hofstad, R. (2017): \emph{Random graphs and complex networks}, Volume I, Cambridge Univeristy Press.

\bibitem{RemcoII} van der Hofstad, R. (2015): \emph{Random graphs and complex networks}, Volume II, available at http://www.win.tue.nl/$\sim$rhofstad.

\bibitem{dist_fv} van der Hofstad, R., Hooghiemstra, G. and van Mieghem, P. (2005): Distances in random graphs with finite variance degrees, \emph{Rand. Struct. Alg.} \textbf{26} 76-123.

\bibitem{dist_iv} van der Hofstad, R., Hooghiemstra, G. and Znamenski, D. (2007): Distances in random graphs with finite mean and infinite variance degrees, \emph{Electr. J. Probab.} \textbf{12} 703-766.

\bibitem{diam} van der Hofstad, R., Hooghiemstra, G. and Znamenski, D. (2009): A phase transition for the diameter of the configuration model, \emph{Internet Math.} \textbf{4} 113-128.

\bibitem{Svante} Janson, S. (2009): The probability that a random multigraph is simple, \emph{Comb. Probab. Computing} \textbf{18}, 205-225.

\bibitem{JL} Janson, S. and Luczak, M. (2009): A new approach to the giant component problem, \emph{Rand. Struct. Alg.} \textbf{34}, 197-216.

\bibitem{Lasse} Leskel\"{a}, L. and Ngo, H. (2017): The impact of degree variability on connectivity properties of large networks, \emph{Internet Math.} \textbf{13}.

\bibitem{MR-95} Molloy, M. and Reed, B. (1995): A critical point for random graphs with a given degree sequence, \emph{Rand. Struct. Alg.} \textbf{6}, 161-179.

\bibitem{MR-98} Molloy, M. and Reed, B. (1998): The size of the giant component of a random graphs with a given degree sequence, \emph{Comb. Prob. Comp.} \textbf{7}, 295-305.

\end{thebibliography}
\end{document}